\documentclass[a4paper, 12pt]{article}
\usepackage{amsmath,amsthm, amssymb, esint}

\newtheorem{theorem}{Theorem}

\theoremstyle{definition}

\theoremstyle{plain}
\newtheorem{thm}{Theorem}[section]
\newtheorem{lemma}[thm]{Lemma}
\newtheorem{cor}[thm]{Corollary}
\newtheorem{fact}[thm]{Fact}
\newtheorem{remark}[thm]{Remark}

\title{Uniform estimates near the initial state for solutions to the two-phase parabolic problem 
\thanks{This work was supported by the Russian Foundation of Basic Research (RFBR) through the grant number
11-01-00825 and by the Russian Federal Target Program 2010-1.1-111-128-033.}
}
\author{D.E.\,Apushkinskaya and N.N.\,Uraltseva \thanks{The second author
thanks, for hospitality and support, the Alexander von Humboldt
Foundation and Saarland University where this work was done.}}

\begin{document}
\maketitle
\begin{abstract}
This paper is devoted to a proof of optimal regularity, near the initial state, for weak solutions to the two-phase parabolic obstacle problem. The approach used here is general enough to allow us to consider the initial data belonging to the class $C^{1,1}$.
\end{abstract}
\bibliographystyle{alpha}
\section{Introduction.}
We consider a weak solution of the two-phase parabolic obstacle problem
\begin{align}
H[u]&=f(u)=\lambda^+\chi_{\{u>0\}}-\lambda^-\chi_{\{u<0\}} \quad \text{in } \quad \mathcal{C}_{10}=B_{10}\times ]0,1],  \label{main-equation} \\
u&=\varphi \quad \text{on} \quad B_{10}, \label{initial-condition}
\end{align}
and $u$ satisfies some boundary conditions on the lateral surface of the cylinder $\mathcal{C}_{10}$. 
Here $H[u]=\Delta u-\partial_t u$ is the heat operator, $\lambda^{\pm}$ are non-negative constants such that $\lambda^++\lambda^->0$, $\chi_E$ is the characteristic function of the set $E$, and  $B_{10}=\{x: |x|<10 \}$. Observe that equation (\ref{main-equation}) is understood in distributional sence. We suppose that a given function $\varphi$ satisfies
\begin{equation} \label{condition-on-varphi}
\varphi \in C^{1,1}\left( B_{10}\right) .
\end{equation}
We suppose also that $\sup\limits_{\mathcal{C}_{10} }|u|\leqslant M$ with $M\geqslant 1$.

\vspace{0.2cm}
For a $C^{1}_x \cap C^0_t$-function $u$ defined in $\mathcal{C}_{10}$ we introduce the following sets:
\begin{align*}
\Omega^{\pm}(u)&=\left\lbrace (x,t)\in \mathcal{C}_{10}: \pm u(x,t)>0\right\rbrace \\
\Lambda (u)&=\left\lbrace (x,t)\in \mathcal{C}_{10}: u(x,t)=|Du(x,t)|=0\right\rbrace \\
\Gamma (u)&=\partial \left\lbrace (x,t) \in \mathcal{C}_{10}: u(x,t) \neq 0\right\rbrace  \cap\, \mathcal{C}_{10} \quad  \text{is the free boundary}.
\end{align*}

We emphasize that in the two-phase case we do not have the property that the gradient vanishes on the free boundary, as it was in the classical one-phase case; this causes difficulties. Therefore, we will distinguish the following parts of $\Gamma$:
$$
\Gamma^{0}(u)=\Gamma (u) \cap \Lambda (u), \qquad \Gamma^{*}(u)=\Gamma (u) \setminus \Gamma^{0} (u).
$$
From \cite{SUW09} it follows that in some suitable rotated coordinate system in $\mathbb{R}^n$ the set $\Gamma^* (u)$ can be locally described as $x_1=f(x_2,\dots,x_n,t)$ with $f\in C^{1, \alpha}$ for any $\alpha <1$.

\vspace{0.2cm} 
\subsection{Background and main result.} 
In this paper we are interested in uniform $L^{\infty}$-estimates near the initial state for the derivatives $D^2u$ and $\partial_t u$ of the function $u$ satisfying (\ref{main-equation})-(\ref{initial-condition}). 

Relative interior estimates were obtained in \cite{SUW09}. The corresponding estimates up to the lateral surface were proved in \cite{U2007} for zero Dirichlet data, and in \cite{AU09} for general Dirichlet data satisfying certain structure conditions, respectively. Unfortunately the proofs presented in \cite{SUW09}, \cite{U2007} and \cite{AU09} do not work near the initial state. To this end, the additional investigation of the behaviour of the solution $u$ close to the initial state is required. 

Speaking about regularity up to the initial state,  we are only aware of the results of \cite{U2007}, \cite{NyPaPo2010}, \cite{S2008} and \cite{Ny08}. In the papers \cite{Ny08} and \cite{S2008} the authors studied the parabolic obstacle problem near the initial state for quasilinear and fully nonlinear equations, respectively. In both cases, the estimates of the second derivatives $D^2u$ were not considered, and hence, only the gradient $Du$ and the time derivative $\partial_t u$ were estimated. The results in \cite{NyPaPo2010} are most close to those obtained here. Indeed, the authors of \cite{NyPaPo2010} considered the parabolic obstacle problem with more general differential operator of Kolmogorov type and established the $L^{\infty}$-estimates of $D^2u$ and $\partial_t u$ under the assumption that the initial data $\varphi$ belongs to the space $C^{2, \alpha}$. 
It remains only to note that the two-phase parabolic problem with the initial data $\varphi \in C^{2,\alpha}$ was studied in \cite{U2007} under the additional structure assumption that $\varphi$ vanishes together with its gradient.

\vspace{0.2cm}
Now we formulate the main result of the paper.

\begin{theorem} \label{main-theorem}
Let $u$ be a weak solution of (\ref{main-equation})-(\ref{initial-condition}) with a function $\varphi$ satisfying the assumption (\ref{condition-on-varphi}). Suppose also that $\sup\limits_{\mathcal{C}_{10}}|u| \leqslant M$.

Then there exists a positive constant $c$ completely defined by $n$, $M$, $\lambda^{\pm}$, and $\varphi$ such that
$$
\text{ess}\, \sup\limits_{\mathcal{C}_{1}}\left\lbrace |D^2u|+|\partial_tu|\right\rbrace \leqslant c.
$$
\end{theorem}


\begin{remark}
The result of Theorem~\ref{main-theorem} is optimal in the sense that we require from the initial function $\varphi$ as much regularity as we want to prove for the solution.
\end{remark}

\begin{remark}
The cylinder $\mathcal{C}_{10}$ is chosen only for simplicity. In fact, the problem (\ref{main-equation})-(\ref{initial-condition}) can be treated in $\mathcal{C}_{1+\delta}$ for arbitrary $\delta >0$. In this case, the constant $c$ in Theorem~\ref{main-theorem} will also depend on $\delta$.
\end{remark}

\begin{remark}
From Theorem~\ref{main-theorem} and the well-known interpolation theory it follows that $Du \in C^{1,1/2}_{x,t}\left( \mathcal{C}_1\right) $.
\end{remark}

The main strategy used in the present paper follows. At first we prove the estimate of the time derivative. We do this in \S 2 with the help of regularizations. The next step is to obtain the estimates of the second derivatives. The analysis of the second derivatives in \S 3 is based essentially on the famous local monotonicity formula due to L.~Caffarelli.


\subsection{Notation.}
Throughout this paper we use the following notation:

\noindent $z=(x,t)$ are points in ${\mathbb R}^{n+1}$, where $x \in
{\mathbb R}^n$ and $t \in {\mathbb R}^1$;

\noindent $|x|$ is the Euclidean norm of $x$ in $\mathbb{R}^n$;

\noindent $\chi _{E }$ denotes the characteristic function of the
set $E\subset {\mathbb R}^{n+1}$;

\noindent $v_+=\max{\{v,0\}}$; \qquad $v_-=\max{\{-v,0\}}$;

\noindent $B_r(x^0)$ denotes the open ball in ${\mathbb R}^n$ with
center $x^0$ and radius $r$; 

\noindent $B_r=B_r(0)$;

\noindent $\mathcal{C}_{r}=B_r \times ]0,1]$;

\noindent $\partial' \mathcal{C}_{r}$ is the parabolic boundary, i.e., the topological boundary minus the top of the cylinder;


\noindent $Q_r(z^0)=Q_r(x^0,t^0)=B_r(x^0) \times ]t^0-r^2,t^0]$. Since our main interest are the estimates near the initial state, the radius $r$ in $Q_r(z^0)$ will be always chosen such that $t^0-r^2=0$.

\noindent $D_i$ denotes the differential operator with respect to
$x_i$; $\partial_t=\frac{\partial}{\partial t}$;

\noindent $D=(D_1,D_2,\dots , D_n)$ denotes the spatial
gradient;

\noindent $D^2u=D(Du)$ denotes the Hessian of $u$;

\noindent
$D_{\nu}$ stands for the operator of differentiation along the direction $\nu\in \mathbb{R}^{n}$, i.e., $|\nu|=1$ and
$$
D_{\nu}u=\sum\limits_{i=1}^n
\nu_iD_i u.
$$

\vspace{0.2cm} \noindent $\|\cdot\|_{p,\,E}$ denotes the norm in
$L^p(E)$, $1<p \leqslant \infty$;

\vspace{0.2cm} \noindent $\fint\limits_{E}\dots$ stands for the average integral over the set $E$, i.e.,
$$
\fint\limits_{E} \dots =\frac{1}{\text{meas}\,\left\lbrace E\right\rbrace }\int\limits_{E} \dots ;
$$

\noindent $\xi=\xi (|x|)$ stands for a time-independent cut-off function belonging $C^2(B_{2})$, having support in $B_{2}$, and satisfying $\xi \equiv 1$ in $B_1$.

\vspace{0.2cm}
\noindent $\xi_{r,x^0}(x)=\xi\left( \dfrac{|x-x^0|}{r}\right)$. It is clear that in the annular $B_{2r}(x^0) \setminus B_{r} (x^0)$ the function $\xi_{r,x^0}$ satisfies the inequalities
\begin{equation} \label{estimates-for-zeta}
|D\xi_{r,x^0} (x)| \leqslant c(n)r^{-1}, \qquad |\Delta \xi_{r,x^0} (x)| \leqslant c(n)r^{-2}.
\end{equation}

For future reference, we introduce the fundamental solution 
\begin{equation} \label{heat kernel}
G(x,t)=\left\lbrace 
\begin{aligned}
\frac{\exp \left( -|x|^2/4t\right) }{(4\pi t)^{n/2}}, \quad &\text{for}\quad t>0,\\
0, \qquad \qquad \quad &\text{for}\quad t \leqslant 0
\end{aligned} \right.
\end{equation}
to the heat equation.

We use letters $M$, $N$, and $c$ (with or without indices) to
denote various constants. To indicate that, say, $c$ depends on some
parameters, we list them in the parentheses: $c(\dots)$. We will
write $c(\varphi)$ to indicate that $c$ is defined by the sum $\|D^2\varphi \|_{\infty, B_{10}}+\|\varphi \|_{\infty, B_{10}}$.



\subsection{Useful facts}
For the readers convenience and for future references we recall and explain some facts.



\begin{fact} \label{sub-caloricity}
Let $u$ be a solution of Equation (\ref{main-equation}), and let $e$ be a direction in $\mathbb{R}^n$. Then
\begin{equation} \label{D_eu-subcaloric}
H \big[\left( D_eu\right)_{\pm}\big] \geqslant 0 \quad \text{in} \quad \mathcal{C}_{10},
\end{equation}
where the inequalities (\ref{D_eu-subcaloric}) are understood in the sense of distributions.
\end{fact}

\begin{proof}
The proof of this assertion can be found in  \cite{SUW09} (see also \cite{AU09}).
\end{proof}

We denote
$$
I(r, v, z^0)=\int\limits_{t^0-r^2}^{t^0}\int\limits_{\mathbb{R}^n}|Dv(x,t)|^2G(x-x^0,t^0-t)dxdt,
$$
where $r\in ]0,1]$, $z^0=(x^0,t^0)$ is a point in $\mathbb{R}^{n+1}$, a function $v$ is defined in the strip $\mathbb{R}^n \times [t^0-1,t^0]$ , and the heat kernel $G(x,t)$ is defined by (\ref{heat kernel}).

To prove the main theorem, we need the following monotonicity formula for pairs of disjointly supported subsolutions of the heat equation.
\begin{fact}\label{monotonicity formula}
Let $\xi:=\xi (|x|)$ be a standard time-independent cut-off function (see Notation), and  
let $h_1$, $h_2$ be nonnegative, sub-caloric and continuous functions in $\mathcal{C}_{2}$, satisfying
$$
h_1(0,1)=h_2(0,1)=0, \qquad h_1(x,t)\cdot h_2(x,t)=0 \quad \text{in}\quad {\mathcal{C}}_{2}.
$$

Then, for $0<r\leqslant 1$  the functional
$$
\begin{aligned}
\Psi (r)=\Psi(r, h_1, h_2, \xi, 0,1 )=&\frac{1}{r^4} I(r, \xi h_1,0,1) I(r, \xi h_2, 0,1)
\end{aligned}
$$
satisfies the inequality
\begin{equation} \label{local-monotonicity-in-B_1}
\Psi (r) \leqslant \Psi (1) + N(n) \|h_1\|^2_{2, {\mathcal{C}}_{2}} \|h_2\|^2_{2, {\mathcal{C}}_{2}}.
\end{equation}
\end{fact}

\begin{proof}
For the proof of this statement we refer the reader to (the proof of) Theorem 1.1.4 \cite{CK98} (see also Theorem 12.12 in \cite{CSa05}).
\end{proof}

\vspace{0.2cm}

\begin{remark} \label{rescaled-monotonicity-formula}
By rescaling one can easily derive  the following modification of the local monotonicity formula (\ref{local-monotonicity-in-B_1}):
$$
\Phi (r, \zeta_R) \leqslant \Phi (R, \zeta_R) +\frac{N(n)}{R^{2n+8}}\|h_1\|^2_{2, Q_{2R}(z^0)} \|h_2\|^2_{2, Q_{2R}(z^0)} \qquad \forall r\leqslant R=\sqrt{t^0}.
$$
Here $\zeta_R(x)=\xi_{R,x^0}(|x|)$ is a standard cut-off function (see Notation) and
\begin{equation} \label{rescaled functional}
\Phi (r, \zeta_R)=\Phi(r, h_1, h_2, \zeta_R, z^0)=\frac{1}{r^4}I(r, \zeta_R h_1, z^0)I(r, \zeta_Rh_2, z^0).
\end{equation}
\end{remark}

\section{Estimate of the time derivative}

For $\varepsilon >0$ we consider the regularized problem
\begin{align}
H[u^{\varepsilon}]&=f^{\varepsilon}(u^{\varepsilon}) \quad \  \text{in } \quad \mathcal{C}_{9},
\label{regularized-eqn}\\
u^{\varepsilon}&=u \qquad \quad \  \text{on} \quad \partial B_{9}\times ]0,1], \label{regularized-BC}\\
u^{\varepsilon}&=\varphi_{ \varepsilon  } \qquad \quad \text{on} \quad B_{9} \times \left\lbrace 0\right\rbrace , \label{regularized-IC} 
\end{align}
where $f^{\varepsilon}$ is a smooth non-decreasing function such that $f^{\varepsilon}(s)=\lambda^+$ as $s \geqslant \varepsilon$ and $f^{\varepsilon}(s)=-\lambda^-$ as $s \leqslant \varepsilon$; while $\varphi_{ \varepsilon }$ is a mollifier of $\varphi$ with the radius depending on the distance to $\partial B_9$ such that 
$$
\sup\limits_{B_{9}}|\varphi-\varphi_{\varepsilon }| \leqslant\varepsilon,
$$ 
 and $u$ satisfies (\ref{main-equation})-(\ref{initial-condition}). 

By the parabolic theory, for each $\varepsilon >0$,  the regularized problem (\ref{regularized-eqn})-(\ref{regularized-IC}) has a solution $u^{\varepsilon}$ with $Du^{\varepsilon}$ and $\partial_t u^{\varepsilon}$ belonging to $L^2\left( \mathcal{C}_{9}\right) $. 

\vspace{0.2cm}

\begin{lemma} \label{uniform-convergence}
Let $\varepsilon >0$, let $u$ satisfy (\ref{main-equation})-(\ref{initial-condition}), and let $u^{\varepsilon}$ be a solution of (\ref{regularized-eqn})-(\ref{regularized-IC}). Then 
\begin{equation} \label{convergence}
\sup\limits_{\mathcal{C}_{9}}|u^{\varepsilon} - u| \leqslant \varepsilon.
\end{equation}
\end{lemma}
\begin{proof}
Setting $w=u^{\varepsilon}-u$ we observe that $w\big|_{\partial '\mathcal{C}_{9}}=u_{\left\lbrace \varepsilon\right\rbrace }-u$, and, consequently, $\left( w-\varepsilon\right)_{+}\big|_{\partial '\mathcal{C}_{9}}=0$.  Then Eqs. (\ref{regularized-eqn}) and (\ref{main-equation}) together with integration by parts provide for arbitrary $t \in ]0,1]$ the following identity
\begin{equation} \label{identity}
\begin{aligned}
\iint\limits_{B_9\times ]0,t]}\left[ f(u)-f^{\varepsilon} (u^{\varepsilon})\right] &\left( w-\varepsilon\right)_+ dz=\iint\limits_{B_9 \times ]0,t]}-H[w] \left( w-\varepsilon\right)_+ dz\\ 
&=\frac{1}{2}\int\limits_{B_9}\left[ (w-\varepsilon)_+\right]^2\bigg|_0^t dx +\iint\limits_{\left\lbrace w>\varepsilon\right\rbrace } |Dw|^2 dz.
\end{aligned}
\end{equation}

Taking into account the relations
$f(u)-f^{\varepsilon} (u^{\varepsilon}) \leqslant 0$ on the set  $\left\lbrace u^{\varepsilon} > \varepsilon\right\rbrace \cup \left\lbrace u<0\right\rbrace $ and $\left( w-\varepsilon\right)_+=0$ on the set $\left\lbrace u^{\varepsilon} \leqslant \varepsilon \right\rbrace \cap \left\lbrace u \geqslant 0\right\rbrace $, we conclude that the left-hand side of 
identity (\ref{identity}) is nonpositive. The latter means that
\begin{equation} \label{leqslant}
\sup\limits_{\mathcal{C}_9}w \leqslant \varepsilon.
\end{equation}

Replacing in identity (\ref{identity}) the term $\left( w-\varepsilon\right)_+$ by $\left( w+\varepsilon\right)_-$ and repeating the above arguments we end up with
\begin{equation} \label{geqslant}
\inf\limits_{\mathcal{C}_9}w \geqslant -\varepsilon.
\end{equation}

Combination inequalities (\ref{leqslant}) and (\ref{geqslant}) finishes the proof.
\end{proof}

\vspace{0.2cm}

We observe also that
$$
\partial_tu^{\varepsilon} \big|_{t=0}=-f^{\varepsilon}(u^{\varepsilon}) \big|_{t=0}+\Delta \varphi_{ \varepsilon }, \qquad x\in B_9.
$$
Thanks to condition (\ref{condition-on-varphi}) we may conclude that $\partial_tu^{\varepsilon} \big|_{t=0}$ are bounded uniformly with respect to $\varepsilon$. Moreover, for each small $\delta >0$ the functions $u^{\varepsilon}$ are smooth in the closure of the cylinder $\mathcal{C}_{9-\delta}$.

With the estimate  (\ref{convergence}) at hands it is easy to check that 
$$
\|\partial_t u^{\varepsilon}\|_{2, \mathcal{C}_{8}} \leqslant c,
$$
and the latter inequality is also uniform with respect to $\varepsilon$. Hence we can easily deduce the following result.

\vspace{0.2cm}

\begin{lemma} \label{regularized estimate}
For each small $\delta >0$ the uniform estimate
\begin{equation} \label{regularized-ut}
\sup\limits_{\mathcal{C}_{8-\delta}}|\partial_t u^{\varepsilon}| \leqslant N_1 (n, M, \lambda^{\pm}, \varphi, \delta )
\end{equation}
holds true for solutions $u^{\varepsilon}$ of  the regularized problem (\ref{regularized-eqn})-(\ref{regularized-IC}).
\end{lemma}

\begin{proof}
We set $v=\partial_t u^{\varepsilon}$. 
It is easy to see that  $v_{\pm}$ are subcaloric in $\mathcal{C}_{9}$.  

Now we may apply the well-known parabolic estimate (see, for example, Lemma~3.1 and Remark~6 from \cite{NU2011}) and get
$$
\sup\limits_{\mathcal{C}_{8-\delta}} v_{\pm} \leqslant N  \sqrt{\fint_{\mathcal{C}_{8}}v_{\pm}^2(z) dz}+
\sup\limits_{B_{8}\times \left\lbrace 0\right\rbrace }v_{\pm},
$$
which implies the desired inequality (\ref{regularized-ut}).
\end{proof}

\vspace{0.2cm}

\begin{remark} \label{remark1}
By virtue of Lemma~\ref{uniform-convergence}, solutions $u^{\varepsilon}$ of the regularized problem (\ref{regularized-eqn})-(\ref{regularized-IC})  converge to $u$ as $\varepsilon \to 0$  uniformly in $\mathcal{C}_{9}$.
\end{remark}

\begin{remark}
It follows from Lemma~\ref{regularized estimate} and Remark~\ref{remark1} that the estimate
\begin{equation} \label{estimate-for-ut}
\sup\limits_{\mathcal{C}_{7}}|\partial_t u| \leqslant N_1(n, M, \lambda^{\pm}, \varphi, 1)
\end{equation}
holds true for a function $u$ satisfying (\ref{main-equation})-(\ref{initial-condition}). Here $N_1$ is the same constant  as in Lemma~\ref{regularized estimate}. 
\end{remark}

\section{Estimates of the second derivatives}

\begin{lemma} \label{estimates-Du-Dphi}
Let the assumptions of Theorem~\ref{main-theorem} hold, 
let  $z^0=(x^0,t^0)$ be a  point in $\mathcal{C}_1$, and let $R=\sqrt{t^0}$.

Then there exists a positive constant $N_2$ completely defined by the values of $n$, $M$, $\lambda^{\pm}$, and the norms of $\varphi$ such that
\begin{equation} \label{pointwise-estimate-u-varphi}
\sup\limits_{B_{2R}(x^0)}|Du(\cdot,t)-D\varphi(\cdot)|\leqslant N_2R \quad \text{for} \quad t\in ]0,t^0].
\end{equation}
\end{lemma}

\begin{remark}
It is evident that $Q_{6R}(z^0) \subset \mathcal{C}_{7}$.
\end{remark}

\begin{proof} First, we observe that the assumption (\ref{condition-on-varphi}) implies 
$\Delta \varphi \in L^{\infty} \left( B_{10}\right) $.
Therefore, for almost all $t \in ]0, 1]$ the difference $u(\cdot, t)-\varphi (\cdot)$ can be regarded in $B_7$ as a solution of an elliptic equation 
\begin{equation} \label{elliptic-solution}
\begin{aligned}
\Delta \left( u(x,t)-\varphi (x)\right) &=F(x)\equiv
 f(u(x,t))+\partial_t u (x,t)-\Delta \varphi (x).
\end{aligned}
\end{equation} 
Due to  estimate (\ref{estimate-for-ut}), the function $|F|$ is bounded by the known constant up to the bottom of the cylinder $\mathcal{C}_{7}$. 

Thus, for a test function $\eta \in W^{1,2}_{0}(B_{6R}(x^0))$ we have the integral identity
\begin{equation} \label{identity-1}
\int\limits_{B_{6R}(x^0)} D(u-\varphi)D\eta dx=\int\limits_{B_{6R}(x^0)} F \eta dx.
\end{equation}
We set in (\ref{identity-1}) $\eta=\left( u-\varphi\right) \zeta_{3R}^2$, where $\zeta_{3R}(x):=\xi_{3R,x^0}(|x|)$ is a standard time-independent cut-off function (see Notation).
Then, after integrating by parts and subsequent application of Young's inequality, identity (\ref{identity-1}) takes the form
\begin{equation} \label{identity-2}
\begin{aligned}
\int\limits_{B_{6R}(x^0)} |D(u-\varphi )|^2\zeta_{3R}^2 dx \leqslant c_1 \int\limits_{B_{6R}(x^0)} &(u-\varphi)\zeta_{3R}^2dx \\ &+c_2 \int\limits_{B_{6R}(x^0)} (u-\varphi)^2 |D\zeta_{3R}|^2dx.
\end{aligned}
\end{equation}
It is easy to see that inequality (\ref{estimate-for-ut}) implies the estimate 
$$
\sup\limits_{B_{6R}(x^0)\times ]0,t^0]} |u-\varphi| \leqslant cR^2.
$$
Putting together  (\ref{identity-2}) and the above inequality we arrive at
\begin{equation} \label{indentity-5}
\int\limits_{B_{3R}(x^0)} |D(u(x,t)-\varphi (x))|^2 dx \leqslant N_2 R^{n+2}, \qquad t\in [0,t^0].
\end{equation}

Now we observe that for almost all $t\in ]0,1]$ and for any  direction $e \in \mathbb{R}^n$ the difference $D_eu-D_e\varphi$ may be considered as a weak solution of the  equation
$$
\Delta \left( D_eu(x,t)-D_e\varphi (x) \right) =-D_eF(x)
$$
 in $B_7$.
The well known results (see \cite{LU68}, \cite{GTr98}) applied to the difference $D_eu-D_e\varphi$ yield the inequality 
$$
\sup\limits_{B_{2R}(x^0)}|D_eu-D_e\varphi| \leqslant c\sqrt{\fint\limits_{B_{3R}(x^0)}|D_eu-D_e\varphi|^2dx}+cR\sup\limits_{B_{3R}(x^0)}|F|.
$$
Combining the last inequality with the estimate (\ref{indentity-5}), we get  (\ref{pointwise-estimate-u-varphi}) and finish the proof. 
\end{proof}

\vspace{0.2cm}

\begin{lemma} \label{estimate-part-of-Phi(R)}
Let the assumptions of Lemma~\ref{estimates-Du-Dphi} hold.  Then 
\begin{equation} \label{first-part-Phi(R)}
\int_0^{R^2}\int_{\mathbb{R}^n}|D^2u(x,t)|^2\zeta_R^2(x)G(x-x^0,t^0-t)dxdt  \leqslant N_3 (n, M, \lambda^{\pm}, \varphi) R^2,
\end{equation}
where $\zeta_R (x):=\xi_{R,x^0}(|x|)$ is a standard time-independent cut-off function (see Notation) and the heat kernel $G$ is defined by the formula (\ref{heat kernel}). 
\end{lemma}

\begin{proof} 

Suppose that $e$ is an arbitrary direction in $\mathbb{R}^n$ if $Du (z^0)=0$ and $e\perp \nu$, where $\nu=Du (z^0) / |D u(z^0)|$, otherwise. From (\ref{condition-on-varphi}), (\ref{pointwise-estimate-u-varphi}) and our choice of $e$ it follows that
\begin{equation} \label{estimate-for-D_evarphi}
\sup\limits_{B_{2R}(x^0)}|D_e\varphi|\leqslant cR.
\end{equation}

According to Fact~\ref{sub-caloricity}, the functions $v=\left( D_eu\right)_{\pm}$ are sub-caloric in $\mathcal{C}_2$, i.e., 
$H[v]\geqslant 0$ in the sense of distributions. 
Since $|Dv|^2+vH[v]=\frac{1}{2}H[v^2]$ we have
\begin{equation}\label{1}
\begin{aligned}
\int\limits_0^{R^2}\int\limits_{B_{2R}(x^0)} &|Dv(x,t)|^2 \zeta_R^2(x)G(x-x^0,t^0-t)dxdt\\ 
&\leqslant
\frac{1}{2}\int\limits_0^{R^2}\int\limits_{B_{2R}(x^0)} H[v^2(x,t)]\zeta_R^2(x)G(x-x^0,t^0-t)dxdt. 
\end{aligned}
\end{equation}
After successive integration the right-hand side of (\ref{1}) by parts we get
\begin{align*}
\int\limits_0^{R^2}\int\limits_{B_{2R}(x^0)}|Dv|^2\zeta_R^2Gdxdt&\leqslant -\int\limits_{B_{2R}(x^0)}\left(\frac{v^2}{2}\zeta_R^2G\right) \bigg|_0^{R^2}dx\\
&+\int\limits_0^{R^2}\int\limits_{B_{2R}(x^0)}\frac{v^2}{2}\zeta_R^2\left[ \partial_tG+\Delta G\right] dxdt \\
&+\int\limits_0^{R^2}\int\limits_{B_{2R}(x^0)} v^2\left[ 2\zeta_R D\zeta_R DG+G|D\zeta_R|^2+G\zeta_R\Delta\zeta_R\right] dxdt\\
&=:I_1+I_2+I_3.
\end{align*}

It is evident that due to (\ref{initial-condition}) we have
\begin{align*}
I_1&\leqslant \int\limits_{B_{2R}(x^0)}\frac{|D_eu(x,0)|^2}{2}\zeta_R^2(x) G(x-x^0,t^0)dx\\
&=\int\limits_{B_{2R}(x^0)}\frac{|D_e\varphi (x)|^2}{2}\zeta_R^2(x) G(x-x^0,t^0)dx \\
&\leqslant cR^2,
\end{align*}
where the last inequality provided by (\ref{estimate-for-D_evarphi}).

Taking into account the relation 
$$
\partial_tG+\Delta G=\partial_t G(x-x^0,t^0-t)+\Delta G(x-x^0,t^0-t)= 0 \quad \text{for}\quad t<t^0,
$$ we conclude that $I_2=0$. 

Finally,  we observe that the integral in $I_3$ is really taken over the set
$E=]0,R^2] \times \left\lbrace B_{2R}(x^0)\setminus B_{R}(x^0)\right\rbrace $. Therefore, in $E$ we have the following estimates for functions involved into~$I_3$
\begin{gather*}
|G (x-x^0, t^0-t)|\leqslant c\frac{e^{-\frac{R^2}{4(R^2-t)}}}{(R^2-t)^{n/2}}\leqslant cR^{-n};\\
\begin{aligned}
|DG(x-x^0,t^0-t) D\zeta (x)|&\leqslant c|G(x-x^0,t^0-t)|\frac{|x-x^0|}{R(R^2-t)}\\
&\leqslant c\frac{e^{-\frac{R^2}{4(R^2-t)}}}{(R^2-t)^{1+n/2}} \leqslant cR^{-n-2}.
\end{aligned}
\end{gather*}
Consequently,
$$
I_3 \leqslant CR^{-n-2}\iint\limits_{E} v^2 dxdt \leqslant C \sup\limits_{E}v^2 \leqslant CR^2,
$$
where the last inequality follows from (\ref{pointwise-estimate-u-varphi}) and (\ref{estimate-for-D_evarphi}).


Thus, collecting all inequalities 
 we get
\begin{equation} \label{21}
\int\limits_0^{R^2}\int\limits_{B_{2R}(x^0)}|Dv|^2\zeta_R^2Gdz \leqslant CR^{2}
\end{equation}

Inequalities (\ref{21}) mean that we obtained the desired integral estimate for all the derivatives $D\left( D_eu\right) $ with $e \perp \nu$. Similar estimate for the derivative $D_{\nu}\left( D_{\nu}u\right) $ follows from (\ref{21}) and Eq. (\ref{main-equation}).
\end{proof}

\vspace{0.2cm}

\begin{cor} \label{cor-3.3}
Let the assumptions of Lemma~\ref{estimates-Du-Dphi} hold, and let $e$ be an arbitrary direction in $\mathbb{R}^n$. Then 
$$
\Phi_e(R)=\Phi(R, ( D_eu)_+, (D_eu)_-, \zeta_R, z^0) \leqslant N_4(n, M, \lambda^{\pm}, \varphi),
$$
where the functional $\Phi$ is defined by the formula (\ref{rescaled functional}), while $\zeta_R(x):=\xi_{R,x^0}(|x|)$ is a standard time-independent cut-off function (see Notation).
\end{cor}

\begin{proof}
The desired inequality follows immediately from the definition (\ref{rescaled functional}) combined with (\ref{first-part-Phi(R)}).
\end{proof}

\vspace{0.2cm}

\begin{lemma} \label{estimate-of-L_2norm-D_eu}
Let the assumptions of Lemma~\ref{estimates-Du-Dphi} hold, let $\nu=Du (z^0)/|Du(z^0)|$, and let $e$ be an arbitrary direction in $\mathbb{R}^n$ if $|Du(z^0)|=0$ and $e\perp \nu$ otherwise. Then
\begin{equation} \label{L_2-norm-of-D_eu}
\|(D_eu)_{\pm}\|^2_{2, Q_{2R}(z^0)}\leqslant N_5(n, M, \lambda^{\pm},\varphi) R^{n+4}.
\end{equation}
\end{lemma}

\begin{proof}
It is evident that inequalities (\ref{pointwise-estimate-u-varphi}) and (\ref{estimate-for-D_evarphi})  imply the estimate (\ref{L_2-norm-of-D_eu}).
\end{proof}

\vspace{0.2cm}

\begin{proof}[Proof of Theorem \ref{main-theorem}]


Let $z^0=(x^0,t^0)$ be an arbitrary point from $\mathcal{C}_1$ such that $|u(z^0)|>0$, and let $e\in \mathbb{R}^n$ be 
the same direction as in Lemma~\ref{estimate-of-L_2norm-D_eu}.

\vspace{0.2cm}

Since $D_eu(z^0)=0$, it follows  that
$$
C(n) |D(D_eu) (z^0)|^4 \leqslant \lim\limits_{r \to 0} \Phi_e (r),
$$
where $\Phi_e (r)=\Phi(r, (D_eu)_{+}, (D_eu)_-, \zeta_R, z^0)$ with the functional $\Phi$ defined by the formula (\ref{rescaled functional}), and $\zeta_R(x)=\xi_{R, x^0}(|x|)$. On the other hand, according to Fact~\ref{monotonicity formula} and Remark~\ref{rescaled-monotonicity-formula} after that, we have for $0<r \leqslant R:=\sqrt{t^0}$ the inequality
$$
\Phi_e(r) \leqslant \Phi_e (R)
+\frac{N(n)}{R^{2n+8}}\|(D_eu)_+\|^2_{2, Q_{2R}(z^0)}\|(D_eu)_-\|^2_{2, Q_{2R}(z^0)}.
$$
Application of Corollary~\ref{cor-3.3} and Lemma~\ref{estimate-of-L_2norm-D_eu} enable us to estimate the right-hand side of the last inequality by the known constant. This means that we proved in the cylinder $\mathcal{C}_1$ the $L^{\infty}$-estimate for all the derivatives $D\left( D_eu\right) (z^0)$ with $e \perp \nu$. It is clear that the corresponding estimate of $D_{\nu}\left( D_{\nu}u\right) (z^0)$ can be now deduced from Eq. (\ref{main-equation}). 

So, we establish the $L^{\infty}$-estimates for $D^2u(z^0)$ for all points $z^0\in \Omega^{\pm}\cap \mathcal{C}_1$, and these estimates do not depend on the distance of $z^0$ from the free boundary $\Gamma (u)$. Since \  $|D^2u|=0$\  almost\  everywhere\  on\  $\Lambda (u)$\  and \ the\  $(n+1)$-dimensional Lebesgue measure of the set $\Gamma^*(u)$ equals zero, we get the uniform estimate of the Lipschitz constant for $Du(\cdot,t)$ for each $t\in [0,1]$.


It remains only to observe that the uniform $L^{\infty}$-estimate of $\partial_t u$ were established in (\ref{estimate-for-ut}). This finishes the proof.
\end{proof}

\bibliography{Bibliography(IS)}

\begin{thebibliography}{SUW09}

\bibitem[AU09]{AU09}
D.~E. Apushkinskaya and N.~N. Uraltseva.
\newblock Boundary estimates for solutions to the two-phase parabolic obstacle
  problem.
\newblock {\em J. Math. Sci. (N. Y.)}, 156(4):569--576, 2009.
\newblock Problems in mathematical analysis. No. 38.

\bibitem[CK98]{CK98}
Luis~A. Caffarelli and Carlos~E. Kenig.
\newblock Gradient estimates for variable coefficient parabolic equations and
  singular perturbation problems.
\newblock {\em Amer. J. Math.}, 120(2):391--439, 1998.

\bibitem[CS05]{CSa05}
Luis Caffarelli and Sandro Salsa.
\newblock {\em A geometric approach to free boundary problems}, volume~68 of
  {\em Graduate Studies in Mathematics}.
\newblock American Mathematical Society, Providence, RI, 2005.

\bibitem[GT01]{GTr98}
David Gilbarg and Neil~S. Trudinger.
\newblock {\em Elliptic partial differential equations of second order}.
\newblock Classics in Mathematics. Springer-Verlag, Berlin, 2001.
\newblock Reprint of the 1998 edition.

\bibitem[LU68]{LU68}
Olga~A. Ladyzhenskaya and Nina~N. Uraltseva.
\newblock {\em Linear and quasilinear elliptic equations}.
\newblock Translated from the Russian by Scripta Technica, Inc. Translation
  editor: Leon Ehrenpreis. Academic Press, New York, 1968.

\bibitem[NPP10]{NyPaPo2010}
Kaj Nystr{\"o}m, Andrea Pascucci, and Sergio Polidoro.
\newblock Regularity near the initial state in the obstacle problem for a class
  of hypoelliptic ultraparabolic operators.
\newblock {\em J. Differential Equations}, 249(8):2044--2060, 2010.

\bibitem[NU11]{NU2011}
A.~I. Nazarov and N.~N. Uraltseva.
\newblock The {H}arnack inequality and related properties of solutions of
  elliptic and parabolic equations with divergence-free lower-order
  coefficients.
\newblock {\em Algebra i Analiz}, 23(1):136--168, 2011.

\bibitem[Nys08]{Ny08}
Kaj Nystr{\"o}m.
\newblock On the behaviour near expiry for multi-dimensional {A}merican
  options.
\newblock {\em J. Math. Anal. Appl.}, 339(1):644--654, 2008.

\bibitem[Sha08]{S2008}
Henrik Shahgholian.
\newblock Free boundary regularity close to initial state for parabolic
  obstacle problem.
\newblock {\em Trans. Amer. Math. Soc.}, 360(4):2077--2087, 2008.

\bibitem[SUW09]{SUW09}
Henrik Shahgholian, Nina Uraltseva, and Georg~S. Weiss.
\newblock A parabolic two-phase obstacle-like equation.
\newblock {\em Adv. Math.}, 221(3):861--881, 2009.

\bibitem[Ura07]{U2007}
N.~N. Uraltseva.
\newblock Boundary estimates for solutions of elliptic and parabolic equations
  with discontinuous nonlinearities.
\newblock In {\em Nonlinear equations and spectral theory}, volume 235--246 of
  {\em Amer. Math. Soc. Transl. Ser. 2}, pages 235--246. Amer. Math. Soc.,
  Providence, RI, 2007.

\end{thebibliography}
\end{document}